\documentclass{amsart}
\usepackage{amssymb,graphicx}

\title[Local Lipschitz geometry of weighted homogeneous surfaces]
{Local Lipschitz geometry of weighted homogeneous surfaces}
\author{L. Birbrair}
\author{A. Fernandes}
\address{Departamento de Matem\'atica, Universidade Federal do Cear\'a, Av.
Mister Hull s/n,Campus do PICI, Bloco 914,CEP: 60.455-760 -
Fortaleza - CE - Brasil.} \email{birb@ufc.br}
\email{alex@mat.ufc.br}
\date{\today}

\def \Q {\mathbb{Q}}

\def \R {\mathbb{R}}

\newtheorem{theorem}{Theorem}[section]
\newtheorem{lemma}[theorem]{Lemma}
\newtheorem{proposition}[theorem]{Proposition}

\theoremstyle{definition}

\newtheorem{example}[theorem]{Example}
\theoremstyle{remark}
\newtheorem{remark}[theorem]{Remark}

\numberwithin{equation}{section}

\begin{document}

\begin{abstract}
  We compute Hoelder Complexes,i.e. the complete bi-Lipschitz invariants,
  for germs of real weighed homogeneous algebraic or semialgebraic surfaces.
\end{abstract}

\maketitle

\section{Introduction}

A basic question of Metric Theory of Singularities is Lipschitz
Classification of Singular Sets. Some recent results of several
authors are devoted to Lipschitz invariants of semialgebraic or
algebraic sets with singularities. (See, for example,
\cite{B},\cite{BB},\cite{F},\cite{KP},\cite{V}).

H\"older Complexes, constructed in \cite{B}, are complete
bi-Lipschitz invariants for germs of semialgebraic surfaces. L\^e
Dung Trang asked the following natural question: what is the
relation between Lipschitz invariants and the algebraic nature of
the semialgebraic sets? In this paper we give a complete answer to
this question for weighted homogeneous surfaces in $\R^n$, i.e. we
compute the exponents in H\"older Complexes of these sets.

In order to compute these exponents we consider weighted homogeneous
singular foliations in $\R^n$ and prove that the corresponding
H\"older Exponents can be computed in terms of orders of contact of
leaves of such a foliation.

We would like to thank Professor L\^e Dung Trang for a very
interesting question and very important comments on a preliminary
version of this paper and Professor Fuensanta Aroca for interesting
discussions.

\section{Preliminaries and main results}

We are going to recall a definition of Canonical H\"older Complex
presented in \cite{B}.

\emph{An Abstract H\"older Complex} is a pair $(\Gamma,\beta)$,
where $\Gamma$ is a finite graph, $E_{\Gamma}$ is the set of edges
of $\Gamma$ and $\beta\colon E_{\Gamma}\rightarrow \Q$ is a rational
valued function such that for each $g\in E_{\Gamma}$, we have
$\beta(g)\geq 1$. A vertex $a\in V_{\Gamma}$ is called \emph{smooth}
or \emph{artificial} if $a$ is connected with exactly two edges and
these edges connect $a$ with exactly two vertices of $\Gamma$. A
vertex $a$ is called a \emph{loop vertex} if $a$ is connected with
exactly two edges and these edges connect $a$ with the same vertex
of $\Gamma$.

An Abstract H\"older Complex $(\Gamma,\beta)$ is called
\emph{Canonical} or \emph{Simplified} if
\begin{enumerate}
\item $\Gamma$ has no artificial vertices;
\item for any loop vertex connected with two edges $g_1$ and $g_2$
we have $\beta(g_1)=\beta(g_2)$.
\end{enumerate}

\emph{The Standard H\"older Triangle} $T_{\beta}$ is a semialgebraic
subset of $\R^2$ defined as follows:
$$T_{\beta}=\{(x,y)\in\R^2 \ : \ 0\leq y\leq x^{\beta}, \ 0\leq x\leq 1\},$$
where $\beta \geq 1$ is a rational number.

A semialgebraic set $X\subset\R^n$ is called a \emph{Geometric
H\"older Complex associated to} $(\Gamma,\beta)$ if
\begin{enumerate}
\item there exists a homeomorphism $\Phi\colon cone(\Gamma)\rightarrow
X$, where $cone(\Gamma)$ is a cone over $\Gamma$;
\item for any edge $g\in E_{\Gamma}$, the image of the set
$cone(g)\subset cone({\Gamma})$ by the map $\Phi$ is
semialgebraically bi-Lipschitz equivalent, with respect to the inner
metric, to $T_{\beta}$, where $\beta=\beta(g)$. If
$\Psi\colon\Phi(cone(g))\rightarrow T_{\beta}$ is the corresponding
bi-Lipschitz map, then $\Psi(x_0)=0$, where $x_0\in X$ is the image
of the vertex of $cone(\Gamma)$ by the map $\Phi$.
\end{enumerate}

\begin{theorem}\cite{B}\label{geometric_holder_complex}
Let $X\subset\R^n$ be a closed semialgebraic set of dimension 2 and
let $x_0\in X$. Then there exists a unique (up to isomorphism)
Canonical H\"older Complex $(\Gamma,\beta)$ such that, for
sufficiently small $\epsilon >0$, $X\cap B(x_0,\epsilon)$ is a
Geometric H\"older Complex, associated to $(\Gamma,\beta)$.
\end{theorem}

Let $a_1\geq a_2\geq\dots\geq a_n$ be a finite sequence of positive
integer numbers. A \emph{Weighted homogeneous foliation}
$\mathcal{F}_{(a_1,\dots,a_n)}$ in $\R^n$ with the weights $a_1\geq
a_2\geq\dots\geq a_n$ is a singular foliation defined as a family of
curves $(t^{a_1}x_1,\dots,t^{a_n}x_n)$, where
$x=(x_1,\dots,x_n)\in\R^n$ and $t\in  (0,+\infty)$. The
\emph{Standard Newton Simplex}, associated to a weighted homogeneous
foliation $\mathcal{F}_{(a_1,\dots,a_n)}$ is the convex hull of the
points $(a_1,0\dots,0), (0,a_2,0,\dots,0),\dots,(0,\dots,0,a_n)$.
All the 1-dimensional faces of the Standard Newton Simplex belong to
subspaces
$$\R^2_{ij}=span\{(0,\dots,a_i,0,\dots,0),(0,\dots,a_j,0,\dots,0)\}.$$
The quotient $\displaystyle\frac{a_i}{a_j} \ (i<j)$ is called the
\emph{direction} of $(i,j)$-1-dimensional face of a Standard Newton
Simplex.

Let $x\neq 0$ be a point in $\R^n$. Denote by $\gamma_x$ the closure
of the leaf of $\mathcal{F}_{(a_1,\dots,a_n)}$ passing through the
point $x$. A set $X\subset\R^n$ is called
\emph{$(a_1,\dots,a_n)$-weighted homogeneous} if, for all $x\in X$,
we have $\gamma_x\subset X$.

\begin{example} Let $f\colon \R^3\rightarrow \R$  be a weighted homogeneous
polynomial with respect to weights $a\geq b\geq c$. Then
$X=f^{-1}(0)$ is an $(a,b,c)$-weighted homogeneous algebraic set.
\end{example}

\begin{example} Let $F\colon \R^m \rightarrow \R^n$ be a polynomial
map with coordinate functions $F=(f_1,\dots,f_n)$ such that
$f_1,\dots,f_n$ are weighted homogeneous polynomials with degrees
$(d_1,\dots,d_n)$. Then $F(R^m)$ is a $(d_1,\dots,d_n)$-weighted
homogeneous semialgebraic set.
\end{example}

\begin{proposition}\label{sing(X)}
Let $X\subset\R^n$ be a semialgebraic $(a_1,\dots,a_n)$-weighted
homogeneous subset. Then $Sing(X)$ is also an
$(a_1,\dots,a_n)$-weighted homogeneous set.
\end{proposition}

\begin{proof} Let us consider a map
$\varphi_t\colon\R^n\rightarrow\R^n$ defined as follows:
$$\varphi_t(x_1,\dots,x_n)=(t^{a_1}x_1,\dots,t^{a_n}x_n).$$
Note that a restriction of this map to $\R^n-\{0\}$ is a
diffeomorphism, for all $t>0$. If $X\subset\R^n$ is a semialgebraic
$(a_1,\dots,a_n)$-weighted homogeneous subset, then, for all $t>0$,
$\varphi_t(X)=X$. Thus, $\varphi_t(Sing(X))=Sing(X)$ and hence
$Sing(X)$ is $(a_1,\dots,a_n)$-weighted homogeneous set.
\end{proof}

A closed semialgebraic set $X\subset\R^n$ is called a
\emph{semialgebraic surface} if $\mathrm{dim} X =2$.

\begin{theorem}\label{2.5} Let $X\subset \R^n$ be a
semialgebraic $(a_1,\dots,a_n)$-weighted homogeneous surface. Let
$(\Gamma,\beta)$ be the Canonical H\"older Complex of $X$ at $0$.
Then, for each $g\in E_{\Gamma}$, $\beta(g)=1$ or is one of the
$(i,j)$-directions of the Standard Newton Simplex associated to
$\mathcal{F}_{(a_1,\dots,a_n)}$.
\end{theorem}

For the case of weighted homogeneous surfaces in $\R^3$, we will
prove the following result.

\begin{theorem}\label{2.6}
Let $X\subset\R^3$ be a semialgebraic $(a_1,a_2,a_3)$-weighted
homogeneous surface. If $0\in X$ is an isolated singular point and
the local link of $X$ at $0$ is connected, then the germ of $X$ at
$0$ is bi-Lipschitz equivalent, with respect to the inner metric, to
a germ at $0$ of a $\beta$-horn, i.e. a surface defined as follows:
$$H_{\beta}=\{(x_1,x_2,y)\in\R^3 \ : \ (x_1^2+x_2^2)=y^{2\beta}\},$$
where $\beta$ is equal to 1 or to $\displaystyle\frac{a_2}{a_3}$.
\end{theorem}

\section{Order contact of semialgebraic arcs. Weighted homogeneous foliations}

Recall that a semialgebraic arc $\gamma$ at a point $x_0\in\R^n$ is
image of a semiagebraic map
$\bar{\gamma}\colon[0,\epsilon)\rightarrow\R^n$ such that
$\bar{\gamma}(0)=x_0$ and $\bar{\gamma }(s)\neq 0$, for $s\neq 0$.
Let $\gamma_1,\gamma_2$ be two semialgebraic arcs at $x_0$. These
arcs can be reparametrized near $x_0$ in the following form:
$$\gamma_i(t)=\{x\in \gamma_i \ : \ \|x-x_0\|=t\}; \ i=1,2.$$
Let $\rho(t)=\|\gamma_1(t)-\gamma_2(t)\|$. Since $\rho$ is a
semialgebraic function, we have
$$\rho(t)=at^{\lambda}+o(t^{\lambda}),$$ where $\lambda$ is a
rational number bigger or equal to 1 and $a>0$. The number $\lambda$
is called the \emph{order of contact} of $\gamma_1$ and $\gamma_2$.
We use the notation $\lambda(\gamma_1,\gamma_2)$.

Set
$$\widetilde{\gamma}_i(t)=\{x\in \gamma_i \ : \ \|x-x_0\|_{max}=t\};
\ i=1,2.$$ Let $\widetilde{\rho}(t)=\|
\widetilde{\gamma}_1(t)-\widetilde{\gamma}_2(t)\|_{max}$. Recall
that $\|x\|_{max}=max\{|x_1|,\dots,|x_n|\}$. Since
$\widetilde{\rho}$ is also a semialgebraic function, we have
$$\widetilde{\rho}(t)=\widetilde{a}t^{\widetilde{\lambda}}+
o(t^{\widetilde{\lambda}}),$$ where $\tilde{\lambda}$ is a rational
number bigger or equal to 1 and $\tilde{a}>0$.

\begin{proposition}\label{comparison1}

The number $\widetilde{\lambda}$, defined above, is equal to
$\lambda(\gamma_1,\gamma_2)$.
\end{proposition}

We are going to prove this proposition in Section
\ref{final_section}.

\begin{proposition}\cite{BF}\label{isosceles}

Let $\gamma_1,\gamma_2$ and $\gamma_3$ be three semialgebraic arcs
at $x_0\in\R^n$. Let $\lambda(\gamma_1,\gamma_2)\geq
\lambda(\gamma_2,\gamma_3)\geq\lambda(\gamma_1,\gamma_3)$, then
$\lambda(\gamma_2,\gamma_3)=\lambda(\gamma_1,\gamma_3)$.
\end{proposition}

The main result of this section is the following.

\begin{theorem}\label{theorem_section_3} Let $\mathcal{F}_{(a_1,\dots,a_n)}$ be a weighted
homogeneous foliation in $\R^n$. For all $x\neq y$ in $\R^n$, the
order contact $\lambda(\gamma_x,\gamma_y)$ is equal to 1 or to a
direction of a 1-dimensional face of the Standard Newton Simplex
associated to $\mathcal{F}_{(a_1,\dots,a_n)}$.
\end{theorem}

\begin{proof} Let us proceed by induction on $n$. First, we consider
the case $n=2$. In this case, all the leaves of this foliation can
be presented in one of the following forms:
\begin{enumerate}

\item $x_1\geq 0$, $x_2=ax_1^{\alpha}$, where $a\in\R$ and
$\displaystyle\alpha=\frac{a_1}{a_2}$;
\item $x_1\leq 0$, $x_2=a|x_1|^{\alpha}$, where $a\in\R$ and
$\displaystyle\alpha=\frac{a_1}{a_2}$;
\item $x_1=0$, $x_2\leq 0$;
\item $x_1=0$, $x_2\geq 0$.
\end{enumerate}

Using Proposition \ref{comparison1} one can show that
$\lambda(\gamma_x,\gamma_y)$ is equal to 1 or to $\alpha$.

Suppose that the statement is true for all the weighted homogeneous
foliations $\mathcal{F}_{(a_1,\dots,a_k)}$ in $\R^k$  for $k<n$.
Consider a foliation $\mathcal{F}_{(a_1,\dots,a_n)}$. Note that the
restriction of $\mathcal{F}_{(a_1,\dots,a_n)}$ to the hyperplane
$x_n=0$ is a weighted homogeneous foliation
$\mathcal{F}_{(a_1,\dots,a_{n-1})}$ in $\R^{n-1}$. Thus, for any two
curves $\gamma_y, \gamma_z$ belongs to the hyperplane, the statement
is true, by the induction hypotheses. Thus, we can suppose that
$\gamma_y, \gamma_z$ are chosen in such a way that
$z=(z_1,\dots,z_n)$ and $z_n\neq 0$. Let $y=(y_1,\dots,y_n)$. If
$y_n=0$, then the unit tangent vector at zero to $\gamma_y$ belong
to the hyperplane $x_n=0$ and the unit tangent vector to $\gamma_z$
does not belong to this hyperplane. Thus,
$\lambda(\gamma_y,\gamma_z)=1$. If $y$ and $z$ belong to different
sides of the hyperplane $x_n=0$, then their unit tangent vectors at
zero cannot coincide. Again, in this case,
$\lambda(\gamma_y,\gamma_z)=1$. Now, we suppose that $z_n>0$ and
$y_n>0$ (the case $z_n<0$ and $y_n<0$ can be treated in the same
way). Consider the parametrization $\bar{\gamma}_y(t)$ of $\gamma_y$
and $\bar{\gamma}_z(t)$ of $\gamma_z$ defined in the beginning of
this section. We have
$$\bar{\gamma}_y(t)=(t^{\frac{a_1}{a_n}}\bar{y}_1,
\dots,t^{\frac{a_{n-1}}{a_n}}\bar{y}_{n-1},t)$$ and
$$\bar{\gamma}_z(t)=(t^{\frac{a_1}{a_n}}\bar{z}_1,
\dots,t^{\frac{a_{n-1}}{a_n}}\bar{z}_{n-1},t)$$ where
$\bar{y}=(\bar{y}_1,\dots,\bar{y}_{n-1},1)$ and
$\bar{z}=(\bar{z}_1,\dots,\bar{z}_{n-1},1)$ are the intersections of
$\gamma_y$ and $\gamma_z$, respectively, with the hyperplane
$x_n=1$. We obtain
$$\widetilde{\rho}(t)=max\{t^{\frac{a_1}{a_n}}|\bar{y}_1-\bar{z}_1|,
\dots,t^{\frac{a_{n-1}}{a_n}}|\bar{y}_{n-1}-\bar{z}_{n-1}| \}.$$
Hence, $\lambda(\gamma_y,\gamma_z)$ can be equal to
$\displaystyle\frac{a_1}{a_n},\dots,\frac{a_{n-1}}{a_n}$. The
theorem is proved.
\end{proof}

\begin{remark} By the construction, it is clear that, for all the
pairs ($i>j$), there exists a pair of curves $\gamma_y$ and
$\gamma_z$ such that
$\displaystyle\lambda(\gamma_y,\gamma_z)=\frac{a_i}{a_j}$.
\end{remark}

\section{H\"older exponents. Horn exponents.}

A semialgebraic surface $X\subset\R^n$ is called a
\emph{$\beta$-H\"older Triangle} at $x_0\in X$ if the germ of $X$ at
$x_0$ is semialgebraically bi-Lipschitz equivalent to a germ of the
standard $\beta$-H\"older Triangle $T_{\beta}\subset\R^2$, with
respect to the inner metric, and the image of the point $x_0$ by the
corresponding bi-Lipschitz map is the point $(0,0)\in\R^2$. The
inverse images of the boundary curves of $T_{\beta}$, containing
$(0,0)$ are called \emph{sides} of the $\beta$-H\"older Triangle
$X$. The number $\beta$ is called the \emph{H\"older Exponent} of
$X$ at $x_0$. We use the notation $\beta(X,x_0)$.

A semialgebraic surface $X\subset\R^n$ is called a
\emph{$\beta$-Horn} at a point $x_0\in X$ if the germ of $X$ at
$x_0$ is semialgebraically bi-Lipschitz equivalent to the germ at
$(0,0,0)\in\R^3$ of the standard $\beta$-Horn, i.e. a semialgebraic
set defined as follows:
$$H_{\beta}=\{(x_1,x_2,y)\in\R^3 \ : \ (x_1^2+x_2^2)=y^{2\beta}\},$$
with respect to the inner metric, and the image of the point $x_0$
by the corresponding bi-Lipschitz map is the point $(0,0,0)\in\R^3$.
The number $\beta$ is called the \emph{Horn Exponent} of $X$ at
$x_0$. We are going to use the same notation $\beta(X,x_0)$.

The following result is useful for calculations of H\"older
Exponents and Horn Exponents.

\begin{theorem}\label{calculation_of_beta_exponent} Let
$X\subset\R^n$ be a semialgebraic surface. Let $x_0\in X$ be a point
such that $X$ is a $\beta$-H\"older Triangle at $x_0$ or a
$\beta$-Horn at $x_0$. Then $\beta(X,x_0)=\inf
\{\lambda(\gamma_1,\gamma_2) \ : \ \gamma_1,\gamma_2 \mbox{ are
semialgebraic arcs on X with } \gamma_1(0)=\gamma_2(0)=x_0\}$.
\end{theorem}

\begin{proof}
We are going to prove the statement for a $\beta$-H\"older Triangle.
The proof for a $\beta$-Horn is the same. Let $X\subset\R^n$ be a
$\beta$-H\"older Triangle at $x_0\in X$. By the main result of
\cite{B} (see also \cite{Ku}), there exists a finite set of
semialgebraic arcs at $x_0$, $\{\gamma_1,\dots,\gamma_k\}$,
$\gamma_i\subset X$ for all $i=1,\dots k$, such that

\begin{enumerate}
\item $\gamma_i,\gamma_{i+1}$ are sides of a $\beta_i$-H\"older
Triangle $X_i\subset X$, where
$\beta_i=\lambda(\gamma_i,\gamma_{i+1})$;
\item $\gamma_j\cap X_i=x_0$ if $j\neq i$ and $j\neq i+1$;
\item $X_i\cap B(x_0,\epsilon)$ is normally embedded in $\R^n$, for
sufficiently small $\epsilon >0$.
\end{enumerate}

By the simplification theorem of \cite{B}, $\beta(X,x_0)=\min
\beta_i$.

Let $\alpha_1,\alpha_2\subset X$ be two semialgebraic arcs at $x_0$.
Since $\alpha_1$ and $\alpha_2$ are semialgebraic, then there exist
two subsets $X_{j_1}$ and $X_{j_2}$, defined above, such that
$\alpha_1\subset X_{j_1}$ and $\alpha_2\subset X_{j_2}$. We can
suppose that $j_1<j_2$. By Proposition \ref{isosceles}, we obtain
$$\lambda(\alpha_1,\alpha_2)=\min\{\lambda(\alpha_1,\gamma_{j_1+1}),
\lambda(\gamma_{j_1+1,\gamma_{j_1+2}}),\dots,\lambda(\gamma_{j_2},\alpha_2)\}.$$
By the same reason,
$$\lambda(\alpha_1,\gamma_{j_1+1})\geq \beta_{j_1}
\mbox{ and } \lambda(\gamma_{j_2},\alpha_2)\geq \beta_{j_2}.$$ By
these three inequalities, we obtain $\lambda(\alpha_1,\alpha_2)\geq
\beta(X,x_0)$.

On the other hand, there exists a pair $\gamma_i,\gamma_{i+1}$, such
that $\beta(X,x_0)=\beta_i=\lambda(\gamma_i,\gamma_{i+1})$.
\end{proof}

\section{ Canonical H\"older Complex for weighted homogeneous surfaces.}

This section is devoted to a proof of Theorem \ref{2.5}. We use
induction on the dimension of the ambient space $\R^n$.

Let $X\subset\R^2$ be a closed semialgebraic surface which is
$(a_1,a_2)$-weighted homogeneous. If $X\neq\R^2$, then $X$ is a
collection of some H\"older Triangles $X_1,\dots,X_p$ such that
$X_i\cap X_j=\{0\}$ if $i\neq j$. By Proposition \ref{sing(X)}, we
have two possibilities:
\begin{enumerate}
\item the boundary curves of $X_i$ belong to $Sing(X)$ and, thus, the
boundary curves $\gamma_1$ and $\gamma_2$ are leaves of the weighted
homogeneous foliation;
\item $X_i$ is a " half " of a weighted homogeneous $\beta$-Horn, in
this case we can also suppose that the boundary curves of $X_i$ are
leaves of this foliation.
\end{enumerate}

Thus, $X_i$ is a $(a_1,a_2)$-weighted homogeneous set. If a
$\beta$-H\"older Triangle $X_i$ intersects with the set $x_2=0$ only
at $\{0\}$, then $\beta(X_i,0)=\frac{a_1}{a_2}$. Otherwise,
$\beta(X_i,0)=1.$ Let us observe that $\beta(\R^2,0)=1$. The first
step of induction is done.

Let $X\subset\R^n$ be a semialgebraic $(a_1,\dots,a_n)$-weighted
homogeneous surface. Let $X_i$ be a $\beta$-H\"older Triangle
corresponding to the Canonical Complex of $X$ at $0$. If $X_i$
belongs to the hyperplane $x_n=0$, then the statement is true, by
the induction hypothesis. Thus, let us suppose that $X_i$ does not
belong to the hyperplane $x_n=0$, i.e. there exists a curve
$\gamma_z\subset X_i$ such that $\gamma_z\cap\{x_n=0\}=\{0\}$. Now,
if $X_i\cap\{x_n=0\}\neq\{0\}$, then there exists a curve
$\gamma_y\subset X_i\cap\{x_n=0\}$. The curves $\gamma_z$ and
$\gamma_y$ have different unit tangent vectors at $0\in\R^n$ and,
thus, $\beta(X_i,0)=1$. Now we consider the case when
$X_i\cap\{x_n=0\}=\{0\}$. We are going to show that, for any pair of
semialgebraic arcs $\alpha_1,\alpha_2\subset X_i$ with same initial
point $0\in\R^n$, there exists a pair of leaves $\gamma_{z_1}$ and
$\gamma_{z_2}$ such that
$\lambda(\alpha_1,\alpha_2)\geq\lambda(\gamma_{z_1},\gamma_{z_2})$.
In order to prove this statement, we need the following lemma.

\begin{lemma}\label{final_lemma} Let $Y\subset\R^n$ be a
$(a_1,\dots,a_n)$ weighted homogeneous surface, such that $Y-\{0\}$
is connected. Let $Y\cap\{x_n=0\}=\{0\}$. Suppose that the section
$Y\cap\{x_n=1\}$ is contained in the plane $\{x_i=r\}$. Then, for
every positive value $\epsilon$, there exists a value $r(\epsilon)$
such that the section $Y\cap\{x_n=\epsilon\}$ is contained  to the
plane $x_i=r(\epsilon)$.
\end{lemma}

\begin{proof} Take $\displaystyle
r(\epsilon)=r{\epsilon}^{\frac{a_i}{a_n}}.$
\end{proof}

Let $M=X_i\cap\{x_n=1\}$. Suppose that there exists an index $k$
such that $M\subset\{x_k=r\}$. Let us consider a projection
$P\colon\R^n\rightarrow\R^{n-1}$ defined as follows:
$$P(x_1,\dots,x_n)=(x_1,\dots,\hat{x_k},\dots,x_n).$$ Observe that,
for $j\neq k$, $P |_{X_j}$ is a bi-Lipschitz map and $P(X_j)$ is a
$(a_1,\dots,\hat{a_k},\dots,a_n)$-weighted homogeneous subset in
$\R^{n-1}$. Then we obtain our statement from the induction
hypotheses.

Now, let us suppose that $M\not \subset\{x_{n-1}=r\}$, for all $r$.
Let $T\colon M\times [0,+\infty)\rightarrow X_i$ be a map defined as
follows:
$$T(x_1,\dots,x_n,t)=(t^{\frac{a_1}{a_n}}x_1,\dots,t^{\frac{a_{n-1}}{a_n}}x_{n-1},t).$$
Clearly, the map $T$ is semialgebraic, injective and surjective, for
$t\neq 0$.

Let $\alpha\colon [0,\rho]\rightarrow X_i $ be a semialgebraic arc
in $X_i$ such that $\alpha(0)=0$, parameterized in such a way that
$\alpha^{\prime}(0)$ exists. Let
$\bar{\alpha}\colon[0,\rho]\rightarrow M\times [0,1]$ be a lifting
of $\alpha$, i.e. $T\circ \bar{\alpha}=\alpha.$ This lifting is
defined as follows: $\bar{\alpha}(s)=T^{-1}(\alpha(s))$, for $s\neq
0$. Since $\bar{\alpha}(s)$ is semialgebraic, then
$\displaystyle\lim_{s\to 0}\bar{\alpha}(s)$ exists and belongs to
$M\times \{0\}.$ By the same reason, $\displaystyle\lim_{s\to
0}\bar{\alpha}^{\prime}(s)$ also exists. Therefore, the arc
$\bar{\alpha}$ near $M\times \{0\}$ can be reparameterized in the
following way: $\bar{\alpha}(t)=(x_1(t),\dots,x_{n-1}(t),t)$. By
definition of the map $T$, we have
$$\alpha(t)=(t^{\frac{a_1}{a_n}}x_1(t),\dots,t^{\frac{a_{n-1}}{a_n}}x_{n-1}(t),t).$$
Clearly, $\displaystyle\frac{d}{dt}|_{t =0}(\alpha(t))$ is not
contained in the hyperplane $x_n=0$.

Let $\alpha_1,\alpha_2\colon [0,\rho]\rightarrow X_i$ be two arcs
such that $\alpha_1(0)=\alpha_2(0)=0.$ The arcs $\alpha_1$ and
$\alpha_2$ can be parameterized as follows:
$$\alpha_1(t)=(t^{\frac{a_1}{a_n}}y_1(t),\dots,t^{\frac{a_1}{a_{n-1}}}y_{n-1}(t),t)
\mbox{ and }
\alpha_2(t)=(t^{\frac{a_1}{a_n}}z_1(t),\dots,t^{\frac{a_1}{a_{n-1}}}z_{n-1}(t),t).$$
We obtain:
$$\|\alpha_1(t)-\alpha_2(t)\|_{max}= max\{t^{\frac{a_i}{a_n}}|y_i(t)-z_i(t)|; \
i=1,\dots,n-1\}.$$ Since $|y_i(t)-z_i(t)|$ is a bounded function and
$a_1\geq\cdots\geq a_{n-1}$, we have
$$\|\alpha_1(t)-\alpha_2(t)\|_{max} = at^{\lambda}+o(t^{\lambda})$$ with
$\lambda\geq \frac{a_{n-1}}{a_n}$.

On the other hand, since $M\not \subset\{x_{n-1}=r\}$, there exist
$y=(y_1,\dots,y_{n-1},1),z=(z_1,\dots,z_{n-1},1)\in M$ such that
$y_{n-1}\neq z_{n-1}$. By Theorem \ref{theorem_section_3}, the
leaves $\gamma_y$ and $\gamma_z$ have the order of contact
$\lambda(\gamma_y,\gamma_z)=\frac{a_{n-1}}{a_n}.$

The theorem is proved. \hfill{$\Box$}

\begin{proof}[Proof of Theorem \ref{2.6}]
Let $X\subset\R^3$ be a semialgebraic $(a_1,a_2,a_3)$-weighted
homogeneous surface with a connected local link at $0$. If
$X\cap\{x_3=0\}\neq\{0\}$, then, by the proof of the Theorem
\ref{2.5}, we obtain that $\beta(X,0)=1$.

Note that $\beta(X,0)$ can be equal to
$\displaystyle\frac{a_1}{a_3}$ only in the case that
$X\cap\{x_3=\epsilon\}$ is totally included in a line
$x_2=r(\epsilon)$. But, since the local link of $X$ at $0$ is
connected, it implies that $X\cap\{x_3=\epsilon\}$ is the set
defined by $x_3=\epsilon$, $x_2=r(\epsilon)$. Then $X$ is a union of
standard leaves of $\mathcal{F}_{(a_1,a_2,a_3)}$ passing through the
points belonging to the straight line $x_3=\epsilon$,
$x_2=r(\epsilon)$. Note, that if $X\cap\{x_3=0\}=\{0\}$ then $X$
cannot be closed.

The case $\displaystyle\beta(X,0)=\frac{a_1}{a_2}$ can occur if, and
only if, $X\subset\{x_3=0\}$. But, in this case, $\beta(X,0)=1$.
\end{proof}

\section{Order comparison lemma}\label{final_section}

Let $K$ be a field of germs of subanalytic functions
$f\colon(0,\epsilon)\rightarrow \R$. Let $\nu\colon K\rightarrow\R$
be a canonical valuation on $K$. Namely, if $f(t)=\alpha
t^{\beta}+o(t^{\beta})$ with $\alpha\neq 0$, we put $\nu(f)=\beta$.

Here we are going to prove a bit more general result such that Lemma
\ref{comparison1} is a partial case of it.

\begin{theorem}\label{final_theorem} Let $\| \cdot \|_S$ be a semialgebraic norm on
$\R^n$. Let $\gamma_1$ and $\gamma_2$ be two semianalytic arcs such
that $\gamma_1(0)=\gamma_2(0)=x_0 \in \R^n$. Let $\gamma_i^S(t)$ be
a parametrization of $\gamma_i$ such that
$\|\gamma_i^S(t)-x_0\|_S=t$, $i=1,2$. Let
$\lambda_S(\gamma_1,\gamma_2)=\nu(
\|\gamma_1^S(t)-\gamma_2^S(t)\|_S)$. Then
$\lambda_S(\gamma_1,\gamma_2)=\lambda(\gamma_1,\gamma_2)$.
\end{theorem}

In order to prove this theorem we need the following lemma.

\begin{lemma}\label{final_lemma}
Let $M\subset\R^n$ be a semialgebraic convex compact subset such
that $0\in Int(M)$. Then, for each small $\epsilon>0$, there exists
a number $\delta >0$ such that, for each pair $x,y\in \partial M$
with $\|x-y\|<\epsilon$, the angle between $x$ and $x-y$ satisfies
the following inequality
$$\delta < \angle(x,x-y)<\pi -\delta.$$
\end{lemma}

\begin{proof} Let $Tang(M)$ be a subset of $\R^n\times \R P^{n-1}$
of the pairs $(x,l)$ such that $l$ is a straight line, $l\cap
Int(M)=\emptyset$ and $x\in \partial M\cap l$. Clearly, $Tang(M)$ is
a compact semialgebraic subset of $\R^n\times \R P^{n-1}$. Let $ang
\colon \R^n\times \R P^{n-1}\rightarrow \R$ be a function defined as
follows:
$$ang(x,l)=\sin (\angle (\overrightarrow{0x},l)).$$
Observe that $\sin (\angle (\overrightarrow{0x},l))$ is a well
defined function. Since $Tang(M)$ is compact, then there exists
$\tilde{\delta}>0$ such that, for all $(x,l)\in Tang(M)$, we have
$ang(x,l)>\tilde{\delta}$.

Let $Tang_{\epsilon}(M)$ be a set of pairs $(x,l)$ where $x\in
\partial M$ and $l$ is a straight line passing through $x$ and some
$y\in \partial M$ such that $\|x-y\|\leq \epsilon$. Observe that
$Tang_{\epsilon}(M)$ is also a compact semialgebraic set. Since the
Hausdorff limit $\displaystyle\lim_{\epsilon \to
0}Tang_{\epsilon}(M)$ belongs to $Tang(M)$, there exists
$\tilde{\epsilon}>0$ such that $\displaystyle
ang(x,l)>\frac{\tilde{\delta}}{2}$, for all $(x,l)\in
Tang_{\tilde{\epsilon}}(M)$. It proves the lemma.
\end{proof}

\begin{proof}[Proof of Theorem \ref{final_theorem}.] Let $x_0$
and let $\gamma_1,\gamma_2$ be arcs satisfying the condition of the
theorem. Let us prove that
$\lambda_S(\gamma_1,\gamma_2)\geq\lambda(\gamma_1,\gamma_2)$.
Suppose that
$\lambda_S(\gamma_1,\gamma_2)<\lambda(\gamma_1,\gamma_2)$. Let
$\gamma_1(t)$ and $\gamma_2(t)$ be points such that
$\|\gamma_1(t)\|=\|\gamma_2(t)\|=t$. Let $\tau=|\gamma_1(t)\|_S$.
Let $\gamma_2^S(\tau)$ be a point on $\gamma_2$ such that
$\|\gamma_2^S(\tau)\|_S=\tau$. Thus, for small $t$, the angle at the
vertex $\gamma_2^S(\tau)$ of the triangle
$\gamma_1(t),\gamma_2(t),\gamma_2^S(\tau)$ must tend to zero. The
line defined by $\gamma_2(t)$ and $\gamma_2^S(\tau)$ tends to the
tangent line of $\gamma_2$ at $0$. Since the ball of radius $\tau$,
with respect to the norm $\|\cdot\|_S$, is a convex set and the
origin belongs to this ball, we obtain a contradiction to Lemma
\ref{final_lemma}.

Using the similar argument we can show that
$\lambda_S(\gamma_1,\gamma_2)\leq\lambda(\gamma_1,\gamma_2)$.
\end{proof}

\end{document}